\documentclass{amsart}
\usepackage[utf8]{inputenc}
\usepackage{amssymb}
\usepackage{amsmath}
\usepackage{color}

\newtheorem{definition}{Definition}
\newtheorem{theorem}{Theorem}
\newtheorem{lemma}{Lemma}
\newtheorem{remark}{Remark}
\newtheorem{corollary}{Corollary}

\title[Smoothness of the topological equivalence]{Smoothness of Topological Equivalence on the Half Line for Nonautonomous Systems  }
\author[Casta\~neda]{\'Alvaro casta\~neda}
\author[Monz\'on]{Pablo Monz\'on}
\author[Robledo]{Gonzalo Robledo}
\address{Universidad de Chile, Departamento de Matem\'aticas. Casilla 653, Santiago, Chile}
\address{Universidad de la Rep\'ublica. Facultad de Ingenier\'ia, C\'odigo Postal 11300, Montevideo,  Uruguay}
\email{castaneda@uchile.cl, grobledo@uchile.cl, monzon@fing.edu.uy}
\thanks{This research has been partially supported by MATHAMSUD program (16-MATH-04 STADE)
and FONDECYT Regular 1170968}
\date{\today}

\begin{document}

\begin{abstract}
We study the differentiability properties of the topological equivalence between a uniformly asymptotically stable linear nonautonomous system and a perturbed system with suitable nonlinearities. For this purpose, we construct a uniformly continuous homeomorphism inspired in the Palmer's one restricted to the positive half line,
providing sufficient conditions ensuring its $C^{r}$--smoothness. Additionally, we study the preservation of the uniform stability properties by this homeomorphism.
\end{abstract}

\maketitle

\section{Introduction}

This work is devoted to study the relation between the solutions of the systems
\begin{equation}
\label{scl}
\dot{x}=A(t)x
\end{equation}
and
\begin{equation}
\label{scnl}
\dot{y}=A(t)y+f(t,y),
\end{equation}
where $A \colon \mathbb{R}^+ \to M(n, \mathbb{R})$  and $f: \mathbb{R}^+ \times \mathbb{R}^n \to \mathbb{R}^n$ have the properties:

\begin{itemize}
    \item [{\bf{(P1)}}] $A(t)$ is continuous and $\sup\limits_{t \in \mathbb{R}^+}||A(t)||= M > 0.$

   \item [{\bf{(P2)}}] The system \eqref{scl} is uniformly asymptotically stable, namely,
   there exist constants $K\geq 1$ and $\alpha>0$ such that its transition
matrix $\Phi(t,s)$ verifies
\begin{equation}
\label{UAS}
||\Phi(t,s)||\leq Ke^{-\alpha(t-s)} \quad \textnormal{for any $t\geq s\geq 0$}.
\end{equation}

    \item [{\bf{(P3)}}] For any $t \geq 0$ and any couple $(y, \bar{y}) \in \mathbb{R}^n \times \mathbb{R}^n$ it follows
\begin{equation}
\label{perturbacion}
\mid f(t,y) - f(t,\bar{y}) \mid \leq \gamma \mid y - \bar{y} \mid \quad \textnormal{and} \quad \mid f(t, y)\mid \leq \mu,
\end{equation}
\end{itemize}
where $|| \cdot ||$ and $|\cdot|$ denote a matrix norm and vector norm respectively.

In the autonomous case, P. Hartman \cite{Hartman0} and D.M. Grobman \cite{Grobman} found a local homeomorphism between the solutions of a nonlinear system and the solutions of its linearization around an hyperbolic equilibrium point. C. Pugh in \cite{Pugh} enhanced the previous result by constructing an explicit and global homeomorphism for the particular case of perturbed linear systems, also requiring hyperbolicity of the equilibrium point.

In the nonautonomous case, K.J. Palmer in \cite{Palmer} extended the Pugh's result using the exponential dichotomy as a  natural version of the hyperbolicity property. In addition, Palmer introduced the concept of topological equivalence, which was generalized by J.L Shi \textit{et. al} in \cite{Shi} as strongly topological equivalence.

Now, in order to present the  aforementioned concepts, we will consider an interval  $J \subset \mathbb{R}.$

\begin{definition}\cite{Coppel, Kloeden}
\label{DE}
The linear system \eqref{scl}
 has an exponential dichotomy property on $J \subset \mathbb{R}$ if there exists a projection $P^{2}=P$
and constants $\bar{K}\geq 1$, $\bar{\alpha}>0$,  such that its fundamental matrix $\Phi(t)$ verifies:
\begin{equation}
\label{ED}
\left\{\begin{array}{rcl}
||\Phi(t)P\Phi^{-1}(s)||&\leq & \bar Ke^{-\bar \alpha(t-s)} \quad \textnormal{for any} \quad t\geq s \quad t,s \in J,\\
||\Phi(t)(I-P)\Phi^{-1}(s)|| &\leq & \bar Ke^{-\bar \alpha(s-t)} \quad \textnormal{for any} \quad s \geq t, \quad t,s \in J.\\
\end{array}\right.
\end{equation}
\end{definition}
\begin{definition}
\label{TopEq}
The systems \textnormal{(\ref{scl})} and \textnormal{(\ref{scnl})} are $J-$topologically equivalent
if there exists a function $H\colon J \times \mathbb{R}^{n}\to \mathbb{R}^{n}$ with the properties
\begin{itemize}
\item[(i)] If $x(t)$ is a solution of \textnormal{(\ref{scl})}, then $H[t,x(t)]$ is a solution
of \textnormal{(\ref{scnl})},
\item[(ii)] $H(t,u)-u$ is bounded in $J \times \mathbb{R}^{n}$,

\item[(iii)]  For each fixed $t\in J$, $u\mapsto H(t,u)$ is an homeomorphism of $\mathbb{R}^{n}$,
\end{itemize}
In addition, the function $G(t,u)=H^{-1}(t,u)$ has properties \textnormal{(ii)--(iii)} and maps solutions of \textnormal{(\ref{scnl})} into solutions of \textnormal{(\ref{scl})}.
\end{definition}

\begin{definition}
\label{Stopeq}
 The systems  \textnormal{(\ref{scl})} and \textnormal{(\ref{scnl})} are $J-$strongly topologically equivalent if they are $J-$topologically equivalents and $H$ is a uniform homeomorphism. Namely, for any $\varepsilon>0$, there exists $\delta(\varepsilon)>0$ such that $|u-\tilde{u}|<\delta$ implies
$|H(t,u)-H(t,\tilde{u})|<\varepsilon$ and $|G(t,u)-G(t,\tilde{u})|<\varepsilon$ for any $t\in J$.
\end{definition}

\begin{definition}
\label{TopEqCr}
The systems \textnormal{(\ref{scl})} and \textnormal{(\ref{scnl})} are $C^{r}$ $J-$topologically equivalent if are $J-$topologically equivalent and $u\mapsto H(t,u)$ is a $C^{r}$--diffeomorphism,
with $r\geq 1$.

\end{definition}

Notice that Definition \ref{DE} implies uniform asymptotic stability when $P = I$ and $J = \mathbb{R}^+$ as in the property {\bf{(P2)}}.
We point out that  definitions \ref{TopEq} and \ref{Stopeq} are slight modifications of the ones introduced by Palmer and Shi, since they considered the particular cases $J = \mathbb{R}$ in \cite{Palmer, Shi} and $J = \mathbb{R}^+$ in \cite{Palmer79}.

In the Section 2 we state our first result, which provides sufficient conditions ensuring the $\mathbb{R}^{+}-$strongly to\-po\-lo\-gi\-cal equivalence between systems \eqref{scl} and \eqref{scnl}. We follow the lines of Palmer \cite{Palmer} which constructs the maps $H$ and $G$ by combining the Green's function associated to the exponential dichotomy with a perturbation $f$ satisfying \textbf{(P3)}, obtaining a $\mathbb{R}$--topological equivalence, this was improved by Shi \cite{Shi} which obtains a $\mathbb{R}$--strong topological equivalence. Palmer and Shi's results are extended in \cite{Jiang,Reinfelds}
by considering non--exponential dichotomies an properties more general than \textbf{(P3)}. Nevertheless to the best of our knowledge, this Green's function approach has always assumed that (\ref{scl}) has an exponential dichotomy in $\mathbb{R}$ rather than in $\mathbb{R}^+$. In this context, our construction has technical differences with the previous works and has interesting consequences from a differentiability point of view.

In the Section 3 we show that if (\ref{scnl}) has an equilibrium, it is unique and the $\mathbb{R}^{+}$--strongly topological equivalence preserves the uniform asymptotic stability between the equilibria of both systems. We also compare this preservation result with a Lyapunov's converse result, which show that both systems admit the same Lyapunov function under smallness assumptions.

The Section 4 states our main Theorem, which provides additional conditions on the smoothness of $f$ which ensures that  (\ref{scl}) and (\ref{scnl}) are $C^{r}$ $\mathbb{R}^{+}-$ topologically equivalent. The restriction to $\mathbb{R}^{+}$ allows us to work in a simpler way than \cite{JDE}. It is important to note that, contrarily to the autonomous context, in the nonautonomous framework there are few results about the global differentiability of maps $H$ and $G$  and we refer the reader to \cite{Cuong} for local results.


\section{Strongly topological equivalence on the half line}



The following Theorem provides a sufficient condition ensuring the $\mathbb{R}^+-$strongly topological equivalence between the systems \eqref{scl} and \eqref{scnl}.

\begin{theorem}
\label{teorema1}
Assume that {\bf{(P1)--\bf{(P3)}}} are satisfied
and
\begin{equation}
\label{hipotesis}
\frac{K\gamma}{\alpha}<1,
\end{equation}
then systems (\ref{scl}) and (\ref{scnl}) are $\mathbb{R}^+-$strongly topologically equivalent.
\end{theorem}

\begin{proof}
In order to make a more readable proof, we will decompose it in several steps. Namely, the step 1 defines two auxiliary systems whose solutions are used in the step 2 to construct the maps $H$ and $G$. To prove that these maps establish a topological equivalence, the properties (i)--(ii) are verified in the step 3, while the uniform continuity is proved in the steps 4 and 5.
\medskip

\noindent\emph{Step 1: Preliminaries.} Let
$t \mapsto x(t,\tau,\xi)$ and $t\mapsto y(t,\tau,\eta)$ be solutions of (\ref{scl}) and (\ref{scnl}) passing through $\xi$ and $\eta$
at $t=\tau$. Now, we will consider the initial value problems
\begin{equation}
\label{pivote2}
\left\{\begin{array}{rcl}
w'&=&A(t)w-f(t,y(t,\tau,\eta))\\
w(0)&=& 0,
\end{array}\right.
\end{equation}
and
\begin{equation}
\label{pivote1}
\left\{\begin{array}{rcl}
z'&=&A(t)z+f(t,x(t,\tau,\xi)+z)\\
z(0)&=& 0.
\end{array}\right.
\end{equation}

By using the variation of parameters formula we have that
\begin{equation}
\label{w-star}
w^{*}(t;(\tau,\eta))=-\int_{0}^{t}\Phi(t,s)f(s,y(s,\tau,\eta))\,ds
\end{equation}
is the unique solution of (\ref{pivote2}). Let $BC(\mathbb{R}^+, \mathbb{R}^n)$ be the Banach space of bounded continuous functions with the supremum norm. Now,  for any couple $(\tau, \xi) \in \mathbb{R}^+ \times \mathbb{R}^n,$ we define the  operator $\Gamma_{(\tau, \xi)} \colon BC(\mathbb{R}^+, \mathbb{R}^n) \to BC(\mathbb{R}^+, \mathbb{R}^n) $ as follows
\begin{equation}
\phi \mapsto \Gamma_{(\tau, \xi)} \phi := \displaystyle \int_0^t \Phi(t,s) f(s,x(s,\tau,\xi) + \phi) \, ds.
\end{equation}

Since $\gamma K/\alpha<1$ it is easy to see by {\bf{(P2)--(P3)}} that the operator $\Gamma_{(\tau, \xi)} $ is a contraction and by the Banach fixed point theorem it follows that
\begin{displaymath}
z^{*}(t;(\tau,\xi))=\int_{0}^{t}\Phi(t,s)f(s,x(s,\tau,\xi)+z^{*}(s;(\tau,\xi))) \, ds
\end{displaymath}
is the unique solution of (\ref{pivote1}).

On the other hand, by uniqueness of solutions it can be proved that
\begin{equation}
\label{identity1}
z^{*}(t;(\tau,\xi))=z^{*}(t;(r,x(r,\tau,\xi))) \quad \textnormal{for any $r\geq 0$},
\end{equation}
and
\begin{equation}
\label{identity2}
w^{*}(t;(\tau,\nu))=w^{*}(t;(r,y(r,\tau,\nu))) \quad \textnormal{for any $r\geq 0$}.
\end{equation}

\noindent\emph{Step 2: Construction of the maps $H$ and $G$.} For any $t\geq 0$ we define the maps $H(t,\cdot)\colon \mathbb{R}^{n}\to \mathbb{R}^{n}$ and
$G(t,\cdot)\colon \mathbb{R}^{n}\to \mathbb{R}^{n}$ as follows:
\begin{displaymath}
\begin{array}{rcl}
H(t,\xi)&:=&\displaystyle  \xi+\int_{0}^{t}\Phi(t,s)f(s,x(s,t,\xi)+z^{*}(s;(t,\xi))\,ds \\\\
&=&\xi + z^{*}(t;(t,\xi)),
\end{array}
\end{displaymath}
and
\begin{equation}
\label{Homeo-G}
\begin{array}{rcl}
G(t,\eta)&:=&\displaystyle \eta -\int_{0}^{t}\Phi(t,s)f(s,y(s,t,\eta))\,ds \\\\
&=&\eta+w^{*}(t;(t,\eta)).
\end{array}
\end{equation}
By using (\ref{identity1}), we can verify that
\begin{displaymath}
\begin{array}{rcl}
H[t,x(t,\tau,\xi)]&=&\displaystyle  x(t,\tau,\xi)+\int_{0}^{t}\Phi(t,s)f(s,x(s,t,x(t,\tau,\xi)+z^{*}(s;(t,x(t,\tau,\xi))))\,ds \\\\
&=& \displaystyle x(t,\tau,\xi)+\int_{0}^{t}\Phi(t,s)f(s,x(s,\tau,\xi)+z^{*}(s;(\tau,\xi)))\,ds \\\\
&=&x(t,\tau,\xi)+z^{*}(t;(\tau,\xi)).
\end{array}
\end{displaymath}

\noindent\emph{Step 3: $H$ and $G$ satisfy properties (i)--(ii) of Definition \ref{TopEq}.}
By (\ref{scl}) and (\ref{pivote1}) combined with the above equality, we have that
$$
\begin{array}{rcl}
\displaystyle \frac{\partial }{\partial t} H[t,x(t,\tau,\xi)] & = & \displaystyle \frac{\partial }{\partial t} x(t, \tau, \xi) + \frac{\partial }{\partial t} z^*(t; (\tau, \xi))\\\\
& = & A(t)x(t,\tau,\xi) + A(t)z^*(t; (\tau, \xi)) + f(t,H[t,x(t,\tau,\xi)])\\\\
& = & A(t)H[t,x(t,\tau,\xi)] + f(t,H[t,x(t,\tau,\xi)]),
\end{array}
$$
then $t\mapsto H[t,x(t,\tau,\xi)]$ is solution of (\ref{scnl}) passing through
$H(\tau,\xi)$ at $t=\tau$. As consequence of uniqueness of solution we obtain
\begin{equation}\label{conj1}
H[t, x(t,\tau, \xi)] = y(t,\tau, H(\tau,\xi)),
\end{equation}
similarly, it can be proved that $t\mapsto G[t,y(t,\tau,\eta)]$ is solution of (\ref{scl}) passing through $G(\tau, \eta)$
at $t=\tau$ and

\begin{equation}\label{conj2}
G[t, y(t,\tau, \eta)] = x(t,\tau, G(\tau,\eta)) = \Phi(t, \tau)G(\tau,\eta),
\end{equation}
and the property (i) follows.
Secondly, by using (\ref{UAS}) and (\ref{perturbacion}) it follows that
\begin{displaymath}
|H(t,\xi)-\xi|\leq  K\mu\int_{0}^{t}e^{-\alpha(t-s)}\,ds \leq \displaystyle \frac{K\mu}{\alpha}
\end{displaymath}
for any $t\geq 0$. A similar inequality can be obtained for
$|G(t,\eta)-\eta|$ and the property (ii) is verified.

\noindent\emph{Step 4: $H$ is bijective for any $t\geq 0$.} We will
first show that $H(t, G(t, \eta)) = \eta$ for any $t\geq 0$. Indeed,
$$
\begin{array}{rcl}
H[t,G[t,y(t,\tau,\eta)]] & = & G[t,y(t,\tau,\eta)] \\\\ & &  + \displaystyle \int_0^t \Phi(t,s) f(s,x(s,t,G[t,y(t,\tau,\eta)])+z^*(s;(t,G[t,y(t,\tau,\eta)]))) \, ds\\\\
& = &  y(t,\tau, \eta)  -\displaystyle \int_0^t  \Phi(t,s)f(s,y(s,\tau, \eta)) \, ds\\\\
&&+\displaystyle \int_0^t \Phi(t,s) f(s,x(s,t,G[t,y(t,\tau,\eta)])+z^*(s;(t,G[t,y(t,\tau,\eta)]))) \, ds.
\end{array}
$$
Let $\omega(t) = | H[t, G[t,y(t,\tau,\eta)]] - y(t,\tau, \eta)| .$ Hence by using {\bf{(P2)}} and {\bf{(P3)}} we have that
$$
\begin{array}{ll}
 \omega(t) & =  \left |\displaystyle \int_0^t \Phi(t,s) \{f(s,x(s,t,G[t,y(t,\tau,\eta)])+z^*(s;(t,G[t,y(t,\tau,\eta)]))) -f(s,y(s,\tau, \eta))\} \, ds \right |\\\\
&\leq  K \gamma  \displaystyle \int_0^t e^{-\alpha(t-s)} |\{x(s,t,G[t,y(t,\tau,\eta)])+z^*(s;(t,G[t,y(t,\tau,\eta)])) -y(s,\tau, \eta)\}| \, ds.
\end{array}
$$
Notice that,
$$x(s,t,G[t,y(t,\tau,\eta)])+z^*(s;(t,G[t,y(t,\tau,\eta)])) = H[s,x(s,t,G[t,y(t,\tau,\eta)])]$$
and recalling that
$$x(s,t,G[t,y(t,\tau,\eta)]) = x(s, \tau, G(\tau, \eta)) = G[s,y(s,\tau, \eta)],$$
we can see
$$ H[s,x(s,t,G[t,y(t,\tau,\eta)])] = H[s,G[s,y(s,\tau,\eta)]].$$
Therefore, we obtain
$$\omega(t) \leq K \gamma  \int_0^t e^{-\alpha(t-s)} \omega(s) \, ds \leq \frac{K \gamma}{\alpha} \displaystyle \sup_{s \in \mathbb{R}^+} \{\omega(s)\} \quad \rm{for \,\, all} \quad t \geq 0.$$

The supremum is well defined by property (i) and the fact that all the solutions of systems \eqref{scl} and \eqref{scnl} are bounded on $\mathbb{R}^+$. Now, we take the supremum on the left side above and due to $K \gamma / \alpha < 1$ it follows that $\omega(t) = 0$ for any $t \geq 0.$ In particular, when we take $t = \tau$ we obtain  $H(\tau, G(\tau, \eta)) = \eta.$

Next, we will prove that $G(t, H(t, \xi)) = \xi.$ In fact, due to (\ref{conj1}) we have that
$$
\begin{array}{rcl}
G[t,H[t,x(t,\tau,\xi)]] & = & H[t,x(t,\tau,\xi)] \\\\ & &  - \displaystyle \int_0^t \Phi(t,s) f(s,y(s,t,H[t,y(x,\tau,\xi)])) \, ds\\\\
&= &  x(t,\tau, \xi) + \\\\
&& \displaystyle \int_0^t  \Phi(t,s) \{f(s,H[s,x(s,\tau,\xi)]) - f(s,y(s,\tau,H(\tau,\xi)))\} \, ds\\\\
& = &  x(t,\tau, \xi).
\end{array}
$$
and taking $t=\tau$ leads to $G(\tau,H(\tau,\xi))=\xi$. In consequence, for any $t\geq 0$, $H$ is a bijection and $G$ is its inverse.

\noindent\emph{Step 5: $H$ and $G$  are uniformly continuous for any fixed $t$.} Firstly, we prove that $G$ is uniformly  continuous.

As stated in \cite[p.823]{Shi}, it can be proved that $\alpha \leq M$.
Now, we construct the auxiliary functions $\theta,\theta_{0}\colon [0,+\infty)\to [0,+\infty)$ defined by
\begin{displaymath}
\theta(t)=1+ K\gamma \left(\frac{e^{(M+\gamma-\alpha)t}-1}{M+\gamma-\alpha}\right) \quad \textnormal{and} \quad
\theta_{0}(t)=\left\{\begin{array}{lcl}
K\gamma & \textnormal{if} & \alpha=M,\\\\
K\gamma\left(\frac{e^{(M-\alpha)t}-1}{M-\alpha}\right)  &\textnormal{if} & \alpha<M.
\end{array}\right.
\end{displaymath}

Now, given $\varepsilon>0$, let us define the constants
\begin{equation}
\label{auxiliares}
L(\varepsilon)=\frac{1}{\alpha}\ln\left(\frac{4\mu K}{\alpha \varepsilon}\right), \quad \theta_{0}^{*}=\max\limits_{t\in [0,L(\varepsilon)]}\theta_{0}(t) \quad \textnormal{and} \quad \theta^{*}=\max\limits_{t\in [0,L(\varepsilon)]}\theta(t).
\end{equation}

We will prove the uniform continuity of $G$ by considering two cases:

\noindent\emph{Case i)} $t\in [0,L(\varepsilon)]$. By {\bf{(P2)}} and {\bf{(P3)}} we can deduce that
\begin{equation}
\label{estimation}
\begin{array}{rcl}
|G(t,\eta) - G(t,\bar{\eta})|& \leq &  |\eta - \bar{\eta} |+  \gamma K e^{-\alpha t}  \displaystyle \int_0^t e^{\alpha s} |y(s,t,\eta) - y(s,t,\bar{\eta})| \, ds.
\end{array}
\end{equation}

Now, by {\bf{(P1)}},{\bf{(P3)}}  and Gronwall's Lemma we obtain that
\begin{equation}
\label{ContCondIni}
\begin{array}{rcl}
 |y(s,t,\eta) - y(s,t,\bar{\eta})| & = & |\eta - \bar{\eta}| + \displaystyle \int_s^t |A(\tau)| |y(\tau, t, \eta) - y(\tau, t, \bar{\eta})| \, d \tau\\\\
 && + \displaystyle \int_s^t |f(\tau, y(\tau, t, \eta)) - f(\tau, y(\tau, t, \bar{\eta}))| \, d \\\\
 &\leq & |\eta - \bar{\eta}| +(M + \gamma) \displaystyle \int_s^t |y(\tau, t, \eta) - y(\tau, t, \bar{\eta})| \, d \tau\\\\
 &\leq&  |\eta - \bar{\eta}| e^{(M+\gamma)(t-s)},
\end{array}
\end{equation}
and we emphasize that this inequality is valid for any $t\geq 0$.

Upon inserting (\ref{ContCondIni}) in (\ref{estimation}), we obtain that
$$
\begin{array}{rcl}
|G(t,\eta) - G(t,\bar{\eta})| & \leq&  \left (1 +  K\gamma e^{(M + \gamma-\alpha) t} \displaystyle \int_0^t e^{-(M+\gamma-\alpha) s} \right )|\eta - \bar{\eta} | \\\\
& = & \displaystyle    \left (1 + K\gamma \left\{\frac{e^{(M+ \gamma-\alpha)t}-1}{M + \gamma-\alpha}\right\} \right )|\eta - \bar{\eta} |\\\\
& \leq & \theta(t)|\eta - \bar{\eta}| \leq \theta^{*}|\eta - \bar{\eta} |.
\end{array}
$$

\noindent\emph{Case ii)} $t>L(\varepsilon)$. By {\bf{(P1)}}--{\bf{(P3)}}, we have that
\begin{equation}
\label{estimacion2}
\begin{array}{rcl}
|G(t,\eta)-G(t,\bar{\eta})| & \leq & \displaystyle |\eta-\bar{\eta}|+2\mu K\int_{0}^{t-L}e^{-\alpha(t-s)}\,ds \\\\
&&\displaystyle +K\gamma\int_{t-L}^{t}e^{-\alpha(t-s)}|y(s,t,\eta)-y(s,t,\bar{\eta})|\,ds\\\\
& = & \displaystyle |\eta-\bar{\eta}|+\frac{2\mu K}{\alpha}e^{-\alpha L}\\\\
&& \displaystyle + K\gamma \int_{0}^{L}e^{-\alpha u}|y(t-u,t,\eta)-y(t-u,t,\bar{\eta})|\,du.
\end{array}
\end{equation}

As in case i), the inequality (\ref{ContCondIni})
implies that
\begin{displaymath}
\begin{array}{rcl}
\displaystyle K\gamma \int_{0}^{L}e^{-\alpha u}|y(t-u,t,\eta)-y(t-u,t,\bar{\eta})|\,du
& \leq & \displaystyle K\gamma \int_{0}^{L}e^{(M+\gamma-\alpha)u}|\eta-\bar{\eta}|\,du\\\\
& = & \displaystyle K\gamma \left\{\frac{e^{(M+\gamma-\alpha)L}-1}{M+\gamma-\alpha}\right\}|\eta-\bar{\eta}|.
\end{array}
\end{displaymath}

Upon inserting the above inequality in (\ref{estimacion2}) and using (\ref{auxiliares}), we have that
\begin{displaymath}
\begin{array}{rcl}
|G(t,\eta)-G(t,\bar{\eta})| &\leq & \displaystyle \left(1+ K\gamma \left\{\frac{e^{(M+\gamma-\alpha)L}-1}{M+\gamma-\alpha}\right\}\right)|\eta-\bar{\eta}|+\frac{2\mu K}{\alpha}e^{-\alpha L}\\\\
&\leq & \displaystyle \theta^{*}|\eta-\bar{\eta}|+\frac{\varepsilon}{2}.
\end{array}
\end{displaymath}

Summarizing, given $\varepsilon>0$, there exists $L(\varepsilon)>0$ and $\theta^{*}>0$ such that:
\begin{displaymath}
|G(t,\eta)-G(t,\bar{\eta})|\leq \left\{\begin{array}{lcl}
\displaystyle \theta^{*}|\eta-\bar{\eta}| &\textnormal{if}& t\in [0,L]\\\\
\displaystyle \theta^{*}|\eta-\bar{\eta}|+\frac{\varepsilon}{2} &\textnormal{if}& t>L,
\end{array}\right.
\end{displaymath}
then it follows that
$$
\forall \varepsilon>0 \,\exists
\delta(\varepsilon)=\frac{\varepsilon}{2\theta^{*}}\quad \textnormal{such that} \quad
|\eta-\bar{\eta}|<\delta \Rightarrow |G(t,\eta)-G(t,\bar{\eta})|<\varepsilon
$$
and the uniform continuity of $G$ follows.

Finally, we will prove that $H$ is uniformly continuous for any $t\geq 0$. As the identity is uniformly continuous, we will only prove that $\xi \mapsto z^{*}(t;(t,\xi))$ is uniformly continuous.

Note that the fixed point $z^{*}(t;(t,\xi))$ can be seen as the uniform limit on $\mathbb{R}^{+}$ of a sequence $z_{j}^{*}(t;(t,\xi))$ defined recursively as follows:
\begin{displaymath}
\left\{\begin{array}{rcl}
z_{j+1}^{*}(t;(t,\xi))&=& \displaystyle \int_{0}^{t}\Phi(t,s)f(s,x(s,t,\xi)+z_{j}^{*}(s;(t,\xi))) \, ds \quad \textnormal{for any $j\geq 1$},\\\\
z_{0}^{*}(t;(t,\xi)) &=&     0
\end{array}\right.
\end{displaymath}

The uniform continuity of each map $\xi \mapsto z_{j}^{*}(t;(t,\xi))$ will be proved inductively by following the lines of \cite{Jiang, Shi}. First, it is clear that
$\xi\mapsto z_{0}^{*}(t;(t,\xi))$ verify this property. Secondly, we will assume that the inductive hipothesis
\begin{displaymath}
\forall\varepsilon>0\,\exists\delta_{j}(\varepsilon)>0 \,\,\textnormal{s.t.} \, |\xi-\bar{\xi}|<\delta_{j} \Rightarrow |z_{j}^{*}(t;(t,\xi))-z_{j}^{*}(t;(t,\bar{\xi}))|<\varepsilon \quad \textnormal{for any $t\geq 0$}.
\end{displaymath}

For the step $j+1$ and given $\varepsilon>0$, we will only consider $\alpha<M$ since the case $\alpha=M$ can be carried out easily. We will use the constants $L(\varepsilon)$ and $\theta_{0}^{*}$ defined in (\ref{auxiliares}) and introduce the notation
$$
\Delta_{j}(t,\xi,\bar{\xi})=z_{j}^{*}(t;(t,\xi))-z_{j}^{*}(t;(t,\bar{\xi})).
$$

As before, we will distinguish the cases $t\in [0,L(\varepsilon)]$ and $t>L(\varepsilon)$. First, for $t\in [0,L(\varepsilon)]$ we use {\bf{(P1)}} combined with the estimation
\begin{equation}
\label{linealCI}
|x(s,t,\xi) - x(s,t,\bar{\xi})| \leq |\xi - \bar{\xi}| e^{M|t-s|},
\end{equation}
and we can
verify that
\begin{displaymath}
\begin{array}{rcl}
|\Delta_{j+1}(t,\xi,\bar{\xi})| &\leq & \displaystyle K\gamma e^{-\alpha t}
\int_{0}^{t}e^{\alpha s}\{|x(s,t,\xi)-x(s,t,\bar{\xi})|+|\Delta_{j}(s,\xi,\bar{\xi})|\}\,ds \\\\
&\leq & \displaystyle K\gamma e^{-\alpha t}
\int_{0}^{t}e^{\alpha s}\{|\xi-\bar{\xi}|e^{M(t-s)}+||\Delta_{j}(\cdot,\xi,\bar{\xi})||_{\infty}\}\,ds\\\\
&\leq & \displaystyle  K\gamma\left\{\frac{e^{(M-\alpha)t}-1}{M-\alpha}\right\}|\xi-\bar{\xi}|+\frac{K\gamma}{\alpha}||\Delta_{j}(\cdot,\xi,\bar{\xi})||_{\infty}\\\\
&\leq& \displaystyle \theta_{0}^{*}|\xi-\bar{\xi}|+\frac{K\gamma}{\alpha}||\Delta_{j}(\cdot,\xi,\bar{\xi})||_{\infty},
\end{array}
\end{displaymath}
where $||\Delta_{j}(\cdot,\xi,\bar{\xi})||_{\infty}=\sup\limits_{t\geq 0}|\Delta(t,\xi,\bar{\xi})|$.

On the other hand, when $t>L(\varepsilon)$, we use \textbf{(P2)} combined with the boundedness of $f$ in $[0,t-L]$ and Lipschitzness in $[t-L,t)$ to deduce that \begin{displaymath}
\begin{array}{rcl}
|\Delta_{j+1}(t,\xi,\bar{\xi})|
&\leq & \displaystyle 2K\mu\int_{0}^{t-L}e^{-\alpha(t-s)}\,ds \\\\
&&\displaystyle + K\gamma
\int_{t-L}^{t}e^{-\alpha(t-s) }\{|x(s,t,\xi)-x(s,t,\bar{\xi})|+|\Delta_{j}(s,\xi,\bar{\xi})|\}\,ds.
\end{array}
\end{displaymath}

By \textbf{(P1)} combined with $u=t-s$, (\ref{auxiliares}) and (\ref{linealCI}), we have that
\begin{displaymath}
\begin{array}{rcl}
|\Delta_{j+1}(t,\xi,\bar{\xi})| &\leq& \displaystyle \frac{2K\mu}{\alpha}e^{-\alpha L}
+\frac{K\gamma}{\alpha}||\Delta_{j}(\cdot,\xi,\bar{\xi})||_{\infty} \\\\
&&
\displaystyle +K\gamma\int_{0}^{L}e^{-\alpha u }|x(t-u,t,\xi)-x(t-u,t,\bar{\xi})|\,du \\\\
&\leq &\displaystyle \frac{2K\mu}{\alpha}e^{-\alpha L}+\frac{K\gamma}{\alpha}||\Delta_{j}(\cdot,\xi,\bar{\xi})||_{\infty} +K\gamma |\xi-\bar{\xi}|\int_{0}^{L}e^{(M-\alpha) u }\,du \\\\
&\leq& \displaystyle
\frac{\varepsilon}{2}+\frac{K\gamma}{\alpha}||\Delta_{j}(\cdot,\xi,\bar{\xi})||_{\infty}+K\gamma \left\{\frac{e^{(M-\alpha)L}-1}{M-\alpha}\right\}|\xi-\bar{\xi}|\\\\
&\leq& \displaystyle
\frac{\varepsilon}{2}+\frac{K\gamma}{\alpha}||\Delta_{j}(\cdot,\xi,\bar{\xi})||_{\infty} + \theta_{0}^{*}|\xi-\bar{\xi}|.
\end{array}
\end{displaymath}

Summarizing, for any $t\geq 0$ it follows that
\begin{displaymath}
|\Delta_{j+1}(t,\xi,\bar{\xi})|\leq
\left\{\begin{array}{lcl}
\displaystyle
 \theta_{0}^{*}|\xi-\bar{\xi}|+\frac{K\gamma}{\alpha} ||\Delta_{j}(\cdot,\xi,\bar{\xi})||_{\infty}& \textnormal{if}& t\in [0,L]\\\\ \displaystyle\frac{\varepsilon}{2}+\frac{K\gamma}{\alpha}||\Delta_{j}(\cdot,\xi,\bar{\xi})||_{\infty} + \theta_{0}^{*}|\xi-\bar{\xi}| &\textnormal{if}& t>L.
\end{array}\right.
\end{displaymath}

Now, for any $\varepsilon>0$ there exists $L(\varepsilon)>0$, $\theta_{0}^{*}>0$ and
\begin{displaymath}
\delta_{j+1}(\varepsilon)=
\min\left\{\delta_{j}(\varepsilon),\frac{\varepsilon}{2\theta_{0}^{*}}\left(1-\frac{K\gamma}{\alpha}\right)\right\}
\end{displaymath}
such that for any $t\geq 0$, we have
\begin{displaymath}
\forall\varepsilon>0\,\exists\delta_{j+1}(\varepsilon)>0 \,\,\textnormal{s.t.} \, |\xi-\bar{\xi}|<\delta_{j+1} \Rightarrow |z_{j+1}^{*}(t;(t,\xi))-z_{j+1}^{*}(t;(t,\bar{\xi}))|<\varepsilon.
\end{displaymath}
and the uniform continuity of $\xi \mapsto z_{j}^{*}(t;(t,\xi)) $ follows for any $j\in \mathbb{N}$.

In order to finish our proof, we choose $N\in \mathbb{N}$ such that for any $j>N$ it follows that
\begin{displaymath}
||z^{*}(\cdot;(\cdot,\xi))-z_{j}^{*}(\cdot;(\cdot,\xi))||_{\infty}<\varepsilon \quad \textnormal{for any} \quad  \xi\in \mathbb{R}^{n},
\end{displaymath}
and therefore, if $|\xi-\bar{\xi}|<\delta_{j}$ with $j>N$, it is true that
\begin{displaymath}
\begin{array}{rcl}
|z^{*}(t;(t,\xi))-z^{*}(t;(t,\bar{\xi}))|&\leq & |z^{*}(t;(t,\xi))-z_{j}^{*}(t;(t,\xi))|+\Delta_{j}(t,\xi,\bar{\xi})\\\\
&&+|z^{*}(t;(t,\bar{\xi}))-z_{j}^{*}(t;(t,\bar{\xi}))|<3\varepsilon,
\end{array}
\end{displaymath}
and the uniform continuity of $\xi\mapsto z^{*}(t;(t,\xi))$ and $\xi\mapsto H(t,\xi)$ follows for any fixed $t\geq 0$.
\end{proof}



\begin{remark}
As we stated in the introduction, the construction of the homeomorphisms $H$
and $G$ and its uniform continuity is inspired in the Palmer \cite{Palmer} and Shi \emph{et.al.} works respectively \cite{Shi}. Nevertheless our restricction
to $\mathbb{R}^{+}$ induces some technical difficulties, for example:

 i) We have not the uniqueness of bounded solutions of \textnormal{(\ref{pivote2})} and (\ref{pivote1}), this fact prompted us to consider the specific couple of initial conditions that allows the mapping between solutions of \textnormal{(\ref{scl})}  and \textnormal{(\ref{scnl})}.

 ii) The proof of the uniform continuity is based in two facts: the continuity of any solution of \textnormal{(\ref{scl})--(\ref{scnl})} with respect to the initial conditions in a compact interval $[0,L]$ and the smallness of the homeomorphisms on $[L,+\infty[$ when $L$ is big enough. This last condition is easily verified when we have an exponential dichotomy on $\mathbb{R}$ but more technical work is needed when we consider the restriction to $\mathbb{R}^{+}$.
\end{remark}

\begin{corollary}
The systems \textnormal{(\ref{scl})--(\ref{scnl})} are $\mathbb R^+-$to\-po\-lo\-gically equivalent under the a\--ssumptions of Theorem \ref{teorema1}.
\end{corollary}

\begin{proof}
The proof is the same of the Theorem. However, in the Step 5, it can be proved that
$$
|G(t,\eta)-G(t,\bar{\eta})|\leq C(t)|\eta-\bar{\eta}|
$$
where
$$
C(t)=\left \{ \begin{array}{rcl}1 + K \gamma \,\, \frac{ 1- e^{(-\alpha+M + \gamma) t}}{\alpha- M - \gamma} \, \, & \textnormal{if}& \, \, \alpha \neq M + \gamma\\\\
1 + K\gamma t \, \, & \textnormal{if}& \, \, \alpha = M + \gamma,
\end{array}
\right.
$$ and the $\mathbb{R}^{+}$--topologically equivalence follows.
\end{proof}

\section{Consequences of the Topological Equivalence  and Stability issues}


It is known that $\bar{y}$ is an equilibrium of (\ref{scnl}) if
\begin{equation}
\label{equilibrio}
A(t)\bar{y}+f(t,\bar{y})=0 \quad \textnormal{for any $t\geq 0$}.
\end{equation}

On the other hand, it is important to emphasize
that the $\mathbb{R}^{+}$--strong topological equivalence between (\ref{scl}) and (\ref{scnl}) does not necessarily imply the existence of an equilibrium for (\ref{scnl}). Indeed, we adapt the example introduced by L. Jiang
in \cite[p.487]{Jiang}
\begin{displaymath}
y'=-y+\frac{1}{5}\left(\frac{\pi}{2}-\arctan(|t|+|y|)\right).
\end{displaymath}

It is easy to see that \textbf{(P1)--(P3)}
are satisfied with $K=M=\alpha=1$, $\gamma=1/5$ and $\mu=\pi/5$. We can see that for any $t_{0}\geq 0$
there exists $y(t_{0})$ such that
$f(t_{0},y(t_{0}))=0$, however the above equation has no equilibria in the sense of (\ref{equilibrio}).

The next results show some properties of the equilibria of the system (\ref{scnl}) when them exist.

\begin{lemma}
\label{UPF}
Assume that \textnormal{\textbf{(P1)--(P3)}}  and condition \eqref{hipotesis} are fulfilled. If \textnormal{(\ref{scnl})} has an equilibrium then it is unique.
\end{lemma}

\begin{proof}
Firstly, notice that
$u\in \mathbb{R}^{n}$ is an equilibrium of (\ref{scnl}) if and only if
\begin{equation}
\label{PFE}
u=\Phi(t,0)u+\int_{0}^{t}\Phi(t,s)f(s,u)\,ds \quad \textnormal{for any $t\geq 0$}.
\end{equation}

In fact, if $u\in \mathbb{R}^{n}$ is an equilibrium of (\ref{scnl}), it follows that $y(t,0,u)=u$ for any $t\geq 0$. Then, by variation of parameters formula, it follows that
$$
u=y(t,0,u)=\Phi(t,0)u+ \int_{0}^{t}\Phi(t,s)f(s,u)\,ds.
$$

Now, if $u\in \mathbb{R}^{n}$ satisfies (\ref{PFE}), we derivate with respect to $t$ obtaining (\ref{equilibrio}).

In order to proof the uniqueness of the equilibrium, we will assume that $u$ and $v$ with $u\neq v$ are equilibria of (\ref{scnl}). Then by (\ref{PFE}) and
defining $w=u-v$ we have
$$
w=\Phi(t,0)w+ \int_{0}^{t}\Phi(t,s)[f(s,u)-f(s,v)]\,ds.
$$

Taking norm and considering \textbf{(P2)--(P3)}, we obtain
$$
|w| \leq Ke^{-\alpha t} |w| + K\gamma e^{-\alpha t} \int_{0}^{t} e^{\alpha s} |w|\,ds.
$$
for every $t\geq 0$. As $|w|>0$ we can deduce that
$$
1 \leq  Ke^{-\alpha t} + \frac{K\gamma}{\alpha} (1-e^{-\alpha t})\quad,\quad t\geq 0.
$$

Then, for arbitrarily large values of $t$, it follows that $1\leq  K\gamma/\alpha$ obtaining a contradiction with (\ref{hipotesis}).
\end{proof}

\begin{remark}
\label{ball}
A simple consequence of Lemma \textnormal{\ref{UPF}} and \textnormal{(\ref{PFE})} is that any equilibrium of \textnormal{(\ref{scnl})} must be in a closed ball centered at the origin with radius $K\mu/\alpha$.
\end{remark}

\begin{remark}
\label{asintotico-0}
If $\bar{y}$ is an equilibrium of \textnormal{(\ref{scnl})} we also can prove that
\begin{displaymath}
\bar{y}=\Phi(t,t_{0})\bar{y}+\int_{t_{0}}^{t}\Phi(t,s)f(s,\bar{y})\,ds \quad \textnormal{for any} \quad t\geq t_{0}\geq 0.
\end{displaymath}
\end{remark}

An interesting question is to determine that if the uniform asymptotical stability of the origin is preserved by the homeomorphism $H(t,\cdot)$ when the equilibrium $\bar{y}$ of (\ref{scnl}) exists.

The following result relates the equilibria of (\ref{scl}) and (\ref{scnl}) with the homeomorphisms  $H(t,\cdot)$ and $G(t,\cdot)$.

\begin{lemma}
\label{HPF}
Assume that \textnormal{\textbf{(P1)--(P3)}} and \textnormal{(\ref{hipotesis})} are fulfilled.
\begin{itemize}
\item[(i)] If $\bar{y}=0$ is equilibrium of \textnormal{(\ref{scnl})}, namely $f(t,0)=0$ for any $t\geq 0$, then
$$
H(t,0)=G(t,0)=0 \quad \textnormal{for any} \quad t\geq 0.
$$
\item[(ii)]
If the system \textnormal{(\ref{scnl})} has a equilibrium  $\bar{y}\neq 0$, then
\begin{displaymath}
\lim\limits_{t\to +\infty}H(t,0)=\bar{y} \quad
\textnormal{and} \quad\lim\limits_{t\to +\infty}G(t,\bar{y})=0.
\end{displaymath}
\end{itemize}
\end{lemma}

\begin{proof}
If $f(t,0)=0$ for any $t\geq 0$ then the system (\ref{pivote2}) becomes (\ref{scl}) and $w^{*}(t;(\tau,0))=0$ for any $t\geq 0$ and by the definition of $G(t,\cdot)$, we have that $G(t,0)=0$ and (i) follows.

Now, let us assume that $\bar{y}\neq 0$ is the unique equilibrium of (\ref{scnl}). Then, the initial value problem (\ref{pivote2}) becomes
\begin{displaymath}
\left\{\begin{array}{rcl}
w'&=&A(t)w-f(t,\bar{y})\\
w(0)&=& 0,
\end{array}\right.
\end{displaymath}
whose solution is given by
\begin{displaymath}
\begin{array}{rcl}
w^{*}(t;(\tau,\bar{y})) &=& \displaystyle         -\int_{0}^{t}\Phi(t,s)f(s,\bar{y})\,ds \\\\
 &=&\displaystyle \int_{0}^{t}\Phi(t,s)A(s)\bar{y}\,ds\\\\
 &=&\displaystyle  -\int_{0}^{t}\frac{\partial}{\partial s}\Phi(t,s)\bar{y}\,ds \\\\
 &=& (\Phi(t,0)-I)\bar{y},
\end{array}
\end{displaymath}
and by using definition of $G(t,\cdot)$ we obtain $G(t,\bar{y})=\Phi(t,0)\bar{y}$ and it follows by \textbf{(P2)} that $\lim\limits_{t\to +\infty}G(t,\bar{y})=0$.

Similarly, if $\xi=0$, the initial value problem (\ref{pivote1}) becomes
\begin{displaymath}
\left\{\begin{array}{rcl}
z'&=&A(t)z+f(t,z)\\
z(0)&=& 0,
\end{array}\right.
\end{displaymath}
which is not parameter dependent and its solution is
\begin{displaymath}
z^{*}(t)=\int_{0}^{t}\Phi(t,s)f(s,z^{*}(s))\,ds.
\end{displaymath}

By the definition, we know that $H(t,0)=z^{*}(t)$, and as $\bar{y}$ is a fixed point, we have that
\begin{displaymath}
H(t,0)-\bar{y}  =  \displaystyle  -\Phi(t,0)\bar{y}+\int_{0}^{t}\Phi(t,s)\{f(s,H(s,0))-f(s,\bar{y})\}\,ds,
\end{displaymath}
which implies that
\begin{displaymath}
\begin{array}{rcl}
|H(t,0)-\bar{y}| & \leq  & \displaystyle  Ke^{-\alpha t}|\bar{y}|+\int_{0}^{t}
Ke^{-\alpha(t-s)}\gamma |H(s,0)-\bar{y}|\,ds.
\end{array}
\end{displaymath}

By Gronwall's inequality we can deduce that
\begin{displaymath}
|H(t,0)-\bar{y}|\leq K|\bar{y}|e^{-\alpha t}+e^{(K\gamma -\alpha)t}.
\end{displaymath}

Finally, by using (\ref{hipotesis}) we obtain $\lim\limits_{t\to +\infty}H(t,0)=\bar{y}$ and the result follows.
\end{proof}


\begin{lemma}
\label{pres-UAS}
Assume that \textnormal{\textbf{(P1)--(P3)}} and \textnormal{(\ref{hipotesis})} are satisfied. If $f(t,0)=0$ for any $t\geq 0$, then the origin is a globally uniformly asymptotically stable solution of \textnormal{(\ref{scnl})}.
\end{lemma}

\begin{proof}
We have to proof that
\begin{equation}
\label{uas2}
\left\{\begin{array}{l}
\forall\varepsilon>0  \wedge \forall c>0 \quad \exists T:T(\varepsilon,c)>0 \quad \textnormal{such that} \\\\
|H[t,x(t,t_{0},\xi)]|<\varepsilon \quad \forall t>t_{0}+T \, \wedge \, \forall \, |H(t_{0},\xi)|<c.
\end{array}\right.
\end{equation}

Firstly, given $c>0$ such that $|H(t_{0},\xi)|<c$, we can deduce that
\begin{displaymath}
\begin{array}{rcl}
\displaystyle |\xi |  &<& \displaystyle c+\int_{0}^{t}|\Phi(t_{0},s)f(s,x(s,t_{0},\xi)+z^{*}(s;(t,\xi))|\,ds  \\\\
 &<& \displaystyle c_{0}(c):=c+\frac{K\mu}{\alpha}.
 \end{array}
\end{displaymath}

We consider $c_{0}(c)$ defined above and by Theorem 4.11 from \cite{Khalil} we know that \textbf{(P2)} is equivalent to
\begin{equation}
\label{uas3}
\left\{\begin{array}{l}
\forall\delta>0  \quad \exists T_{0}:T_{0}(\delta,c_{0})>0 \quad \textnormal{such that} \\\\
|x(t,t_{0},\xi)|<\delta \quad \forall t>t_{0}+T_{0} \, \wedge \, |\xi|<c_{0}.
\end{array}\right.
\end{equation}

Secondly, by strongly topological equivalence combined with the fact $H(t,0)=0$ for any $t\geq 0$, we know that
\begin{equation}
\label{ste-ua}
\forall \varepsilon>0\,\,\exists\widehat{\delta}(\varepsilon)>0 \quad \textnormal{s.t.} \quad |x(t,t_{0},\xi)|<\widehat{\delta} \Rightarrow |H[t,x(t,t_{0},\xi)]|<\varepsilon.
\end{equation}

Now, given $\varepsilon>0$ we define
$$
T(\varepsilon,c):=T_{0}(\widehat{\delta}(\varepsilon),c_{0}(c)),
$$
where $\widehat{\delta}(\varepsilon)$ is from (\ref{ste-ua}) while $T_{0}$  is from (\ref{uas3}). By using (\ref{uas3}) and (\ref{ste-ua}) it follows that
\begin{displaymath}
\forall\varepsilon>0 \, \exists T(\varepsilon,c)>0 \quad \textnormal{s.t.} \quad t>t_{0}+T \Rightarrow |H[t,x(t,t_{0},\xi)]|<\varepsilon,
\end{displaymath}
and (\ref{uas2}) is verified.
\end{proof}

\begin{theorem}
\label{uas4}
Assume that \textnormal{\textbf{(P1)--(P3)}} and \textnormal{(\ref{hipotesis})} are satisfied. If \textnormal{(\ref{scnl})} has a unique equilibrium $\bar{y}\neq 0$, then it is  globally uniformly asymptotically stable.
\end{theorem}

\begin{proof}
Let us consider $z=y-\bar{y}$, then $z$ is solution of the system
\begin{equation}
\label{trasladado}
z'=A(t)z+g(t,z)  \quad \textnormal{with} \quad g(t,z)=f(t,z+\bar{y})-f(t,\bar{y}).
\end{equation}

Notice that $g$ satisfies \textbf{(P2)} with constants $\gamma$ and $2\mu$.  Then by Theorem \ref{teorema1}, the systems (\ref{scl}) and (\ref{trasladado}) are $\mathbb{R}^{+}$--strongly topologically equivalent. Finally, as $g(t,0)=0$ for any $t\geq 0$, the result follows from Lemma \ref{pres-UAS}.
\end{proof}



From now on, we assume that $f(t,0)=0$ for any $t\geq 0$ and we will contrast the results of  Theorems \ref{teorema1} and \ref{uas4} with a Lyapunov stability result for the system \eqref{scnl}. As (\ref{scl}) is uniformly asymptotically stable, the next Lemma recalls the existence of a quadratic Lyapunov function having certain bounds (see, \emph{e.g.}, \cite[Ch.4]{Khalil}).

\begin{lemma}[Lyapunov's converse result]
\label{lemma-forma_cuadratica_scl}
Let us consider the linear system  \eqref{scl} satisfying \textnormal{\textbf{(P1)--(P2)}}. Given a bounded, continuous, positive definite $n\times n$ matrix function $Q(t)$, with constants $q^{-}$ and $q^{+}$ such that
$$
0<q^{-} I\leq Q(t)\leq q^{+}  I.
$$
Then, there exists a bounded, continuous, positive definite $n\times n$ matrix function $P(t)$ such that
\begin{equation}
\label{cotas-lyap}
0<p^{-}I\leq P(t)\leq  p^{+}I\quad\mathrm{and}\quad -\dot P(t)= A^T(t)P(t) + P(t)A(t) + Q(t),
\end{equation}
with constants $p^{-}=q^{-}/2M$ and $p^{+} = K^2 q^{+}/2\alpha$.
\end{lemma}
\begin{proof}
We define
$$
P(t) = \int_t^{+\infty} X^T(\tau,t)Q(\tau)X(\tau,t)d\tau
$$
Then
\begin{displaymath}
\begin{array}{rcl}
x_0^TP(t)x_0 &=& \displaystyle \int_t^{+\infty} x_0^{T} \, X^{T}(\tau,t) \, Q(\tau) \,X(\tau,t) \,x_0 \, d \tau \\\\
 &\leq & \displaystyle \int_t^{+\infty}  q^{+} \|X(\tau,t)\|^2|x_0|^2\\\\
&\leq&  \displaystyle
 \int_t^{+\infty} q^{+} K^2e^{-2\alpha(\tau-t)}|x_0|^2=\frac{K^2 q^{+}}{2\alpha}|x_0|^2.
 \end{array}
\end{displaymath}

Similarly,
\begin{eqnarray*}
x_0^TP(t)x_0 \geq& \displaystyle \int_t^{+\infty} q^{-} |x(\tau,t,x_0)|^2 \geq
 \int_t^{+\infty} q^{-} e^{-2M(\tau-t)}|x_0|^2=\frac{ q^{-}}{2M}|x_0|^2.
\end{eqnarray*}

By defining $p^{-}=q^{-}/2M$ and $ p^{+} =K^2q^{+}/2\alpha$, the left part of (\ref{cotas-lyap}) is verified.

Finally, in order to obtain the right part of (\ref{cotas-lyap}), we see that
$$
\dot P(t)=\int_t^{+\infty} \frac{\partial }{\partial t}X^T(\tau,t)Q(\tau)\frac{\partial }{\partial t}X(\tau,t)d\tau + \int_t^{+\infty} X^T(\tau,t)Q(\tau)X(\tau,t)d\tau-Q(t),
$$
and using $\frac{\partial X}{\partial t}(\tau,t)=-A(t)X(\tau,t)$, the result follows.
\end{proof}

A classical result of Lyapunov stability theory shows that a quadratic Lyapunov function for a linear system may be also work for a Lipschitz nonlinear perturbation satisfying some smallness properties as follows.

\begin{lemma}\label{lemma-forma_cuadratica_scnl} Suppose that $f(t,0)=0$ for any $t\geq 0$. The quadratic Lyapunov function of the linear system (\ref{scl}) can be extended to the nonlinear system (\ref{scnl}) if
$$
\gamma<\frac{\alpha}{K^2}.
$$
\end{lemma}

\begin{proof}
Let us consider a positive definite, bounded matrix function $Q(t)$ such that
$$
0< q^{-} I\leq Q(t)\leq  q^{+}  I.
$$

Let $P(t)$ be the continuous, bounded, positive definite matrix solution of equation
$$
-\dot P(t)= A^T(t)P(t) + P(t)A(t) + Q(t).
$$

We consider the quadratic function $V(t,x)=x^TP(t)x$ as a Lyapunov candidate function for the nonlinear system. The derivative along the trajectories gives
\begin{displaymath}
\begin{array}{rcl}
\dot V(t,x)&=&\dot x^TP(t)x + x^T\dot P(t)x+x^TP(t)\dot x \\\\
&=&x^T\left[A(t)^TP(t)+P(t)A(t) +\dot P(t)\right]x + 2x^TP(t)f(t,x)\\\\
&=& -x^TQ(t)x + 2x^TP(t)f(t,x)\\\\
&\leq & - q^{-}|x|^2 + 2|x|\|P(t)\||f(t,x)|\\\\
&\leq & - q^{-}|x|^2 + 2\gamma \ p^{+}|x|^2\leq -\left(q^{-} -2\gamma p^{+}\right)|x|^2.
\end{array}
\end{displaymath}

Then, if
$$
\gamma\leq \frac{q^{-}}{2p^{+}}=\frac{q^{-}}{\ q^{+}}\frac{\alpha}{K^2}\leq \frac{\alpha}{K^2}
$$
we have a global Lyapunov function for the nonlinear system.
\end{proof}

\begin{remark} Lemma \ref{lemma-forma_cuadratica_scnl} is a result about the preservation of the uniform asymptotic stability which can be compared
with Theorem \ref{teorema1} and Lemma \ref{pres-UAS}. Indeed we know that the condition $\gamma < \frac \alpha K$ ensures the $\mathbb R^+$-strong topological equivalence between both systems and also that the global asymptotical stability is preserved by the homeomorphisms. We emphasize that our condition (\ref{hipotesis}) is less strict  that the one obtained by the Lyapunov function in Lemma \ref{lemma-forma_cuadratica_scnl}.

\end{remark}

\section{Smoothness of the Homeomorphisms $H$ and $G$}

Throughout this section we will assume that $x\mapsto f(t,x)$ is at least $C^{1}$ for any fixed $t$. Hence, it is well known (see \emph{e.g.} \cite[Chap. 2]{Coddington}) that
$\partial y(t,\tau,\eta)/\partial \eta$ satisfies the matrix differential equation
\begin{equation}
\label{MDE1}
\left\{
\begin{array}{rcl}
\displaystyle \frac{d}{dt}\frac{\partial y}{\partial\eta}(t,\tau,\eta)&=&\displaystyle \{A(t)+Df(t,y(t,\tau,\eta))\}\frac{\partial y}{\partial \eta}(t,\tau,\eta),\\\\
\displaystyle \frac{\partial  y}{\partial\eta}(\tau,\tau,\eta)&=&I.
\end{array}\right.
\end{equation}

The following result shows the $C^{r}$ differentiability of the map $G$ constructed previously
\begin{theorem}
Assume that {\bf{(P1)--\bf{(P3)}}}
and \textnormal{(\ref{hipotesis})}
are satisfied. If $y\mapsto f(t,y)$ is $C^{r}$ with $r\geq 1$ for any fixed $t\geq 0$,
 then \textnormal{(\ref{scl})} and \textnormal{(\ref{scnl})} are $C^{r}$ $\mathbb{R}^+-$ topologically equivalent.
\end{theorem}

The proof of this result will be a consequence of the following Lemma and Remarks.
\begin{lemma}
\label{derivabilidad-1}
Under the assumptions of Theorem \ref{teorema1}, if $f$ is $C^1$ then $\eta \mapsto G(t,\eta)$ is a $C^{1}$ diffeomorphism for any fixed $t\geq 0$. Moreover, its Jacobian matrix is
\begin{equation}
\label{jacobiano-G}
\frac{\partial G}{\partial \eta}(t,\eta)=\Phi(t,0)\frac{\partial y(0,t,\eta)}{\partial \eta}.
\end{equation}
\end{lemma}

\begin{proof}
It is well known that (see \emph{e.g.}, Theorem 4.1 from \cite[Ch.V]{Hartman}) if $y\mapsto f(t,y)$ is $C^{r}$ with $r\geq 1$, then the map
$\eta \mapsto y(t,\tau,\eta)$ is also $C^{r}$ for any  fixed couple $(t,\tau)$. Then, as $y\mapsto f(t,y)$ is $C^{1}$, it follows that $y\mapsto Df(t,y)$ and  $\eta \mapsto \partial y/\partial \eta$ are continuous.
This allows to calculate the first partial derivatives of the map $\eta\mapsto G(t,\eta)$ for any $t\geq 0$ as follows
\begin{equation}
\label{derivada-parcial}
\frac{\partial G}{\partial \eta_{i}}(t,\eta)=e_{i}-\int_{0}^{t}\Phi(t,s)Df(s,y(s,t,\eta))\frac{\partial y}{\partial \eta_{i}}(s,t,\eta)\,ds \quad (i=1,\ldots,n),
\end{equation}
which implies that the partial derivatives exists and are continuous for any fixed $t\geq 0$, then $\eta \mapsto G(t,\eta)$ is $C^{1}$.

By using the identity $\Phi(t,s)A(s)=-\frac{\partial }{\partial s}\Phi(t,s)$ combined with
(\ref{MDE1}) we can deduce that for any $t\geq 0$, the Jacobian matrix is given by
\begin{equation}
    \label{machine}
\begin{array}{rcl}
\displaystyle\frac{\partial G}{\partial\eta}(t,\eta)&=&  \displaystyle I-\int_{0}^{t}\Phi(t,s)Df(s,y(s,t,\eta))
\frac{\partial y}{\partial\eta}(s,t,\eta)\,ds \\\\
&=&I - \displaystyle \int_{0}^{t}\frac{d}{ds}\left\{\Phi(t,s)\frac{\partial y}{\partial \eta}(s,t,\eta)\right\}\,ds \\\\
&=&\displaystyle \Phi(t,0)\frac{\partial y(0,t,\eta)}{\partial \eta},
\end{array}
\end{equation}
and Theorems 7.2 and 7.3 from  \cite[Ch.1]{Coddington}  imply that
$Det \frac{\partial G(t,\eta)}{\partial \eta}>0$ for any $t\geq 0$.

Summarizing, we have that $\eta \mapsto G(t,\eta)$ is $C^{1}$ and its Jacobian matrix has a non vanishing determinant. In addition, let us recall that
$$
G(t,\eta)=\eta + w^{*}(t;(t,\eta)),
$$
where $w^{*}(t;(t,\eta))$ is given by (\ref{w-star}), we can deduce that
$|G(t,\eta)|\to +\infty$ as $|\eta|\to +\infty$, due to $|w^{*}(t;(t,\eta))|\leq K\mu/\alpha$ for any $(t,\eta)$.

Therefore, by Hadamard's Theorem (see \emph{e.g} \cite{Plastock, Radulescu}), we conclude that $\eta \mapsto G(t,\eta)$ is a global diffeomorphism for any fixed $t\geq 0$.
\end{proof}

\begin{remark}
\label{Rademacher}
It is interesting to point out that the computation of the partial derivatives is remarkably simple when considering $\mathbb{R}^+\times\mathbb{R}^{n}$ as domain of $G$. In the case when the domain is $\mathbb{R}\times\mathbb{R}^{n},$ we refer to \cite{JDE} for details.
\end{remark}

\begin{remark}
\label{sofit}
In particular, if $y\mapsto f(t,y)$ is $C^{2}$, we can verify that
the second derivatives $\partial^{2}y(s,\tau,\eta)/\partial \eta_{j}\partial\eta_{i}$  satisfy
 the system of differential equations
\begin{equation}
\label{MDE15}
\left\{
\begin{array}{rcl}
\displaystyle \frac{d}{dt}\frac{\partial^{2}y}{\partial\eta_{j}\partial\eta_{i}}&=&\displaystyle
\{A(t)+Df(t,y)\}\frac{\partial^{2}y}{\partial\eta_{j}\partial\eta_{i}}
+D^{2}f(t,y)\frac{\partial y}{\partial\eta_{j}}\frac{\partial y}{\partial\eta_{i}} \\\\
\displaystyle \frac{\partial^{2}y}{\partial\eta_{j}\partial\eta_{i}}  &=&0,
\end{array}\right.
\end{equation}
with $y=y(t,\tau,\eta)$, for any $i,j=1,\ldots,n$. By using (\ref{derivada-parcial}) and (\ref{MDE15}) we have
\begin{displaymath}
\begin{array}{rcl}
\displaystyle \frac{\partial^{2} G}{\partial\eta_{j}\partial \eta_{i}}(t,\eta)&=&\displaystyle -\int_{0}^{t}\Phi(t,s)D^{2}f(s,y(s,t,\eta))\frac{\partial y}{\partial \eta_{j}}(s,t,\eta)\frac{\partial y}{\partial \eta_{i}}(s,t,\eta)\,ds\\\\
& & \displaystyle
-\int_{0}^{t}\Phi(t,s)Df(s,y(s,t,\eta))\frac{\partial^{2} y(s,t,\eta)}{\partial\eta_{j}\partial\eta_{i}}\,ds\\\\
&=&\displaystyle -\int_{0}^{t}\frac{d}{ds}\left\{\Phi(t,s)
\frac{\partial^{2} y(s,t,\eta)}{\partial\eta_{j}\partial\eta_{i}}\right\}\,ds\\\\
&=&\displaystyle  \Phi(t,0)\frac{\partial^{2} y(0,t,\eta)}{\partial\eta_{j}\partial\eta_{i}}.

\end{array}
\end{displaymath}

\end{remark}

Notice that we can obtain the same expression for the second partial derivatives by using directly (\ref{machine}) instead of Remark \ref{sofit}.
In addition, as $\eta \to y(0,t,\eta)$ is $C^{r}$ when $f(t,\cdot)$ is $C^{r}$, we use the identity (\ref{machine}) to deduce that the $m$--th partial derivatives of $\eta \mapsto G(t,\eta)$ for any fixed $t\geq 0$ are given by
\begin{displaymath}
\frac{\partial^{|m|} G(t,\eta)}{\partial\eta_{1}^{m_{1}}\cdots \partial \eta_{n}^{m_{n}}}=\Phi(t,0)\frac{\partial^{|m|} y(0,t,\eta)}{\partial\eta_{1}^{m_{1}}\cdots \partial \eta_{n}^{m_{n}}}, \quad \textnormal{where $|m|=m_{1}+\ldots+m_{n}\leq r$},
\end{displaymath}
and the Theorem follows.

\section{Conclusions}
In this work, we have obtained a sufficient condition for the $\mathbb R^+-$ strong topological equivalence between systems \eqref{scl} and \eqref{scnl}. We have also confronted this result with the existence of a quadratic Lyapunov function for the linear system and its extension to the nonlinear one. Finally, we proved some results about differentiability of the homeomorphism constructed between systems \eqref{scl} and \eqref{scnl}, remarking the simpleness of the proof.


\begin{thebibliography}{99}

\bibitem{JDE} \'A. Casta\~neda and G. Robledo. Differentiability of Palmer's linearization theorem and converse result for density functions. \emph{J. Differential Equations} \textbf{259} (2015), 4634--4650.

\bibitem{Coddington} E. Coddington and N. Levinson.
\emph{Theory of Ordinary Differential Equations} (Mc Graw--Hill: New York, 1955).

\bibitem{Coppel} W. Coppel.
\emph{Dichotomies in Stability Theory}.
Lecture Notes in Mathematics (Berlin: Springer, 1978).

\bibitem{Cuong}
L.V. Cuong, T.S. Doan, S. Siegmund A Sternberg Theorem for Nonautonomous Differential Equations. \emph{J. Dynam. Differential Equations} (2018) doi.org/10.1007/s10884-017-9629-8


\bibitem{Grobman}
D.M. Grobman. Homeomorphism of systems of differential equations, \emph{Dokl. Akad. Nauk SSSR} \textbf{128} (1959), 880--881 (Russian).

\bibitem{Hartman0}
P. Hartman. On local homeomorphisms of Euclidean spaces, \emph{Bol. Soc. Mat. Mexicana} \textbf{5} (1960), 220--241.

\bibitem{Hartman} P. Hartman.
\emph{Ordinary Differential Equations}
(SIAM: Philadelphia, 2002).

\bibitem{Jiang} L. Jiang. Generalized exponential dichotomy and global linearization.
\emph{J. Math. Anal. Appl.} \textbf{315} (2006), 474--490.

\bibitem{Khalil} H. Khalil.
\emph{Nonlinear Systems} (Prentice Hall: Upper Saddle River NJ, 1996).

\bibitem{Kloeden} Kloeden P.E., Rasmussen M.
\emph{Nonautonomous Dynamical Systems} (American Mathematical Society: Providence RI, 2011).


\bibitem{Palmer} K.J. Palmer. A generalization of Hartman’s linearization theorem. \emph{J. Math. Anal. Appl.} \textbf{41}
(1973), 753--758.

\bibitem{Palmer79}
K.J. Palmer. The structurally stable linear systems on the half-line are those with exponential dichotomies. \emph{J. Differential Equations} \textbf{33} (1979), 16--25.

\bibitem{Plastock}
R. Plastock. Homeomorphisms between Banach spaces. \emph{Trans. Amer. Math. Soc.}
\textbf{200} (1974), 1691--7183.

\bibitem{Pugh} C. Pugh.
On a theorem of Hartman,
\emph{American Journal of Mathematics} \textbf{91} (1969), 363--367.

\bibitem{Radulescu}M. Radulescu, S. Radulescu. Global inversion theorems and applications to
differential equations. \emph{Nonlinear Anal}. \textbf{4} (1980), 951--965.

\bibitem{Reinfelds} A. Reinfelds, D. Steinberga.
Dynamical equivalence of quasilinear equations, \emph{International Journal of Pure and Applied Mathematics} \textbf{98} (2015), 355--364.

\bibitem{Shi}
J.L.Shi, K.Q Xiong.
On Hartman's linearization theorem and Palmer's linearization theorem.
\emph{J. Math. Anal. Appl.} \textbf{192} (1995), 813--832.

\bibitem{Xia}
Y-H. Xia, R. Wang, K.I.  Kou, D. O'Regan.
On the linearization theorem for nonautonomous differential equations.
\emph{Bull. Sci. Math.} \textbf{139} (2015), 829--846.

\end{thebibliography}
\end{document}